\documentclass[11pt]{article}

 \usepackage{mathrsfs}
\usepackage{amsmath}
\usepackage{amsthm}
\usepackage{amssymb}
\usepackage{amsfonts}
\usepackage{epsf}
\usepackage{graphicx}
\usepackage{hyperref}
\newtheorem{theorem}{Theorem}[section]
\newtheorem{remark}[theorem]{Remark}
\usepackage{amsmath}

\newcommand{\N}{\mathbb{N}}

\newcommand{\PI}{\textrm{P}_\textrm{\small{I}}}
\newcommand{\PII}{\textrm{P}_\textrm{\small{II}}}
\newcommand{\PIII}{\textrm{P}_\textrm{\small{III}}}
\newcommand{\PIV}{\textrm{P}_\textrm{\small{IV}}}
\newcommand{\PV}{\textrm{P}_\textrm{\small{V}}}
\newcommand{\PVI}{\textrm{P}_\textrm{\small{VI}}}

\begin{document}
\title{The recurrence coefficients of semi-classical Laguerre polynomials and the fourth
 Painlev\'e equation}

\author{Galina Filipuk\footnotemark[1],\quad Walter Van Assche\footnotemark[2],
\quad Lun Zhang\footnotemark[2]}
\date{\today}

\maketitle
\renewcommand{\thefootnote}{\fnsymbol{footnote}}
\footnotetext[1]{Faculty of Mathematics, Informatics and Mechanics,
University of Warsaw, Banacha 2, Warsaw, 02-097, Poland. E-mail:
filipuk@mimuw.edu.pl} \footnotetext[2]{Department of Mathematics,
Katholieke Universiteit Leuven, Celestijnenlaan 200B, B-3001 Leuven,
Belgium. E-mail: \{Walter.VanAssche, lun.zhang\}@wis.kuleuven.be}

\begin{abstract}

We show  that the coefficients of the three-term recurrence relation
for orthogonal polynomials with respect to a semi-classical
extension of the Laguerre weight satisfy the fourth Painlev\'e
equation when viewed as functions of one of the parameters in the
weight. We compare different approaches to derive this result,
namely, the ladder operators approach, the isomonodromy deformations
approach and combining the Toda system for the recurrence
coefficients with a discrete equation. We also discuss a relation
between the recurrence coefficients for the Freud weight and the
semi-classical Laguerre weight and show how it arises from the
B\"acklund transformation of the fourth Painlev\'e equation.

\end{abstract}

\section{Introduction}

One of the most important properties of orthogonal polynomials is
the three-term recurrence relation. For a sequence $(p_n)_{n \in
\N}$ of orthonormal polynomials with respect to a positive measure
$\mu$ on the real line
\begin{equation}\label{eq: orthonormality}
\int p_n(x)p_k(x)\,d\mu(x)=\delta_{n,k},
\end{equation}
where $\delta_{n,k}$ is the Kronecker delta,
 this relation takes the following form:
\begin{equation}\label{3 term orthonormal}
 x p_n(x)=a_{n+1}p_{n+1}(x)+b_np_n(x)+a_np_{n-1}(x)
\end{equation}
with the recurrence coefficients given by the   following integrals
\begin{gather}\label{recu:orthonormal}
a_n=\int x p_n(x)p_{n-1}(x)\,d\mu(x),\qquad b_n=\int
xp_n^2(x)\,d\mu(x).
\end{gather}
Here the integration is over the support $S\subset\mathbb{R}$ of the
measure $\mu$ and it is assumed that $p_{-1}=0$.

One can also associate the monic orthogonal polynomials $P_n(x)$ of
degree $n$ in $x$ with the measure $\mu$, namely,
\begin{equation} \label{pn-formula}
P_n(x) = x^n + \textsf{p}_1(n) x^{n-1} + \cdots,
\end{equation}
such that
\begin{equation} \label{orthogonality monic}
\int P_m(x) P_n(x)\, d\mu (x) = h_n \delta_{m,n}, \quad h_n>0, \quad
m,n=0,1,2,\cdots.
\end{equation}
The three-term recurrence relation now reads
\begin{equation} \label{recurrence}
x P_n(x) = P_{n+1}(x) + \alpha_n P_n(x) + \beta_n P_{n-1}(x),
\end{equation}
where
\begin{equation}\label{recu:monic}
\alpha_n=\frac{1}{h_n}\int x P_n^2(x)\,d\mu(x),\qquad
\beta_n=\frac{1}{h_{n-1}}\int x P_n(x)P_{n-1}(x)\,d\mu(x),
\end{equation}
and the initial condition is taken to be $\beta_0 P_{-1}:=0$.

The recurrence coefficients can  be expressed in terms of
determinants containing the moments of the orthogonality measure
\cite{Chihara}. For classical orthogonal polynomials (Hermite,
Laguerre, Jacobi) one knows these recurrence coefficients explicitly
in contrast to non-classical weights.  A useful characterization of
classical polynomials is the Pearson equation
$$ [\sigma(x)w(x)]'=\tau(x)w(x), $$ where $\sigma $ and $\tau$ are polynomials
satisfying $\deg\sigma\leq 2$ and $\deg\tau= 1$, and
$d\mu(x)=w(x)dx$. Semi-classical orthogonal polynomials are defined
as orthogonal polynomials for which the weight function satisfies a
Pearson equation for which $\deg \sigma
> 2$ or $\deg \tau \neq 1$. See Hendriksen and van Rossum \cite{HvR}
and Maroni \cite{Mar}. The recurrence coefficients of semi-classical
weights obey nonlinear recurrence relations, which, in many cases,
can be identified as discrete Painlev\'e equations; see \cite{BV}
and the references therein.

In this paper we consider polynomials orthogonal on $\mathbb{R}^+$
with respect to the semi-classical Laguerre weight
\begin{equation}
\label{eq:weight} w(x)=w(x;t)=x^{\alpha}e^{-x^2+tx}, \qquad x \in
\mathbb{R}^+,
\end{equation}
with $\alpha>-1$ and $t\in\mathbb{R}$.
We show that the corresponding recurrence coefficients are related
to the fourth Painlev\'e equation $\PIV$ for the function $q=q(z)$,
which is given by
\begin{equation*}
q''=\frac{q'^2}{2q}+\frac{3q^3}{2}+4z q^2+2(z^2-A)q+\frac{B}{q}.
\eqno(P4)
\end{equation*}
The solutions of the fourth Painlev\'e equation have no movable
critical points. The fourth Painlev\'e equation is among the six
well-known Painlev\'e equations, whose solutions are often referred
to as nonlinear special functions due to many important applications
in mathematics and mathematical physics; cf.~\cite{Clarkson, GrLSh,
Noumi}.

We will apply different approaches to derive the fourth Painlev\'e
equation for the recurrence coefficients of the semi-classical
Laguerre polynomials. In particular, we shall use the ladder
operators approach, the isomonodromy deformations approach and the
Toda system for the recurrence coefficients combined with a discrete
equation derived in \cite{BV}. A similar comparison of the methods
is given in \cite{Forrester}, where the recurrence coefficients are
related to the solutions of the fifth Painlev\'e equation. Another
main objective of the paper is  to see how the properties of the
orthogonal polynomials are related to properties of transformations
of the Painlev\'e equation. In particular, by using the Toda system
we show that the discrete equation in \cite{BV} can be obtained from
a B{\"a}cklund transformation of the fourth Painlev\'e equation.
Finally we shall deal with recurrence coefficients associated with
the Freud weight and discuss their connections with the fourth
Painlev\'e equation.

\section{Derivation of the fourth Painlev\'e equation for the recurrence coefficients}

\subsection{The discrete equations and Toda system}\label{sec:Toda}
For the semi-classical Laguerre weight given in \eqref{eq:weight},
we have the following discrete equations for the recurrence
coefficients.
\begin{theorem}\label{thm:recurrence}\cite{BV}
The recurrence coefficients $a_n$ and $b_n$ in the three-term
recurrence relation (\ref{3 term orthonormal}) associated with the
weight \eqref{eq:weight} are given by $2a_n^2 = y_n+n+\alpha/2$ and
$2b_n = t-\sqrt{2}/x_n$, where $(x_n,y_n)$ satisfy
\begin{equation}\label{eq:recurrence}
  \begin{cases}
   x_n x_{n-1}=\displaystyle\frac{y_n+z_n}{y_n^2-\alpha^2/4},  \\
   y_n+y_{n+1}=\displaystyle\frac{1}{x_n}\left(\frac{t}{\sqrt{2}}-\frac{1}{x_n}\right),
  \end{cases}
\end{equation}
and $z_n=n+\alpha/2$.
\end{theorem}

It is shown in \cite{BV} that the system (\ref{eq:recurrence}) can
be obtained from an asymmetric Painlev\'e d$\PIV$ equation by a
limiting process. The proof is based on a Lax pair for the
associated orthogonal polynomials.

If  the positive measure  is given by $\exp(tx)\,d\mu(x)$ on the
real line, where $t$ is a real parameter (assuming   that the
moments exist for all $t\in \mathbb{R}$), then  the coefficients of
the orthogonal polynomials depend on $t$ and satisfy the Toda system
\cite{Moser}, \cite[\textsection 2.8, p. 41]{Ismail}
 (see also \cite{our} for more details).

\begin{theorem}\label{proposition:Toda}
The recurrence coefficients  $a_n(t)$ and $b_n(t)$ of monic
polynomials which are orthogonal  with respect to
$\exp(tx)\,d\mu(x)$ on the real line  satisfy the Toda  system
\begin{equation}\label{Toda}
\begin{cases}
    (a_n^2)' = a_n^2 (b_n-b_{n-1})   \\
     b_n' = a_{n+1}^2 -a_n^2.
\end{cases}
\end{equation}
The  initial conditions $a_n(0)$ and $b_n(0)$ correspond to the
recurrence coefficients of the orthogonal polynomials for the
measure $\mu$.
\end{theorem}

In what follows, we shall use the systems (\ref{eq:recurrence}) and
(\ref{Toda}) to derive the fourth Painlev\'e equation $\PIV$. Later
on we show that system (\ref{eq:recurrence}) can be obtained from a
B{\"a}cklund transformation of $\PIV$.

By substituting the expressions for $a_n=a_n(t)$ and $b_n=b_n(t)$ in
terms of $x_n=x_n(t)$ and $y_n=y_n(t)$ into the system (\ref{Toda}),
we can find, using the system (\ref{eq:recurrence}), that
\begin{align}
\label{y0}y_{n+1}=\frac{x_n^2y_n+\sqrt{2}x_n'}{x_n^2}
,\;\;x_{n-1}=-\frac{2(2y_n+2n+\alpha)}{x_n(\alpha^2-4y_n^2)},\;\;y_n=\frac{\sqrt{2}t
x_n-2\sqrt{2}x_n'-2}{4x_n^2},\end{align} where $'$ denotes the
differentiation $d/dt$. As a result we obtain a second order
nonlinear differential equation for the function $x_n$:
$$x_n''=\frac{3}{2}\frac{x_n'^2}{x_n}+\frac{1}{4}\alpha^2
x_n^3-\frac{x_n}{8}(t^2-4-8n-4\alpha)+\frac{t}{\sqrt{2}}-\frac{3}{4x_n}.$$
Substituting
\begin{equation}\label{change}
x_n=-\frac{\sqrt{2}}{q(z)},\qquad t=2z,
\end{equation}
we get the fourth Painlev\'e
equation $\PIV$ 
for the function $q=q(z)$ with parameters
\begin{equation}\label{parameters}A=1+2n+\alpha,\qquad B=-2\alpha^2.\end{equation}

\subsection{Ladder operators approach: preliminaries}

The ladder operators for orthogonal polynomials have been derived by
many authors with a long history, we refer to
\cite{BonanC,BonanLN,Bonan,ci2,sh} and the references therein for a
quick guide. Nowadays the ladder operators approach   has been
successfully applied to show the connections of the Painlev\'e
equations and recurrence coefficients of certain orthogonal
polynomials; cf. \cite{chen+dai,chen+zhang,Dai+Zhang}.

Assume that the weight function $w$ vanishes at the endpoints of the
orthogonality interval. Following the general set-up (cf.
\cite{ci2}), the lowering and raising ladder operators for monic
polynomials $P_{n}(x)$ in (\ref{orthogonality monic}) are given by
\begin{align}
\left( \frac{d}{dx} + B_n(x) \right) P_n(x) & = \beta_n A_n(x)
P_{n-1}(x), \label{ladder1} \\ \left( \frac{d}{dx} - B_n(x) -
\textsf{v}'(x) \right) P_{n-1}(x) & = - A_{n-1}(x) P_n(x)
\label{ladder2}
\end{align}
with $$\textsf{v}(x):=-\ln w(x)$$ and  \begin{align} A_n(x) & :=
\frac{1}{h_n} \int \frac{\textsf{v}'(x) - \textsf{v}'(y)}{x-y} \
[P_n(y)]^2 w(y) dy, \label{an-def}\\ B_n(x) & := \frac{1}{h_{n-1}}
\int \frac{\textsf{v}'(x) - \textsf{v}'(y)}{x-y} \ P_{n-1}(y) P_n(y)
w(y) dy. \label{bn-def}
\end{align}
Note that $A_n(x)$ and $B_n(x)$ are not independent, but satisfy the
following supplementary conditions \cite[Lemma 3.2.2 and Theorem
3.2.4]{Ismail}.
\begin{theorem}
The functions $A_n(x)$ and $B_n(x)$ defined  by (\ref{an-def}) and
(\ref{bn-def}) satisfy 
$$
B_{n+1}(z) + B_n(z)  = (z- \alpha_n) A_n(z) -
\textup{\textsf{v}}\,'(z), \eqno(S_1)
$$
$$
1+ (z- \alpha_n) [B_{n+1}(z) - B_n(z)] = \beta_{n+1} A_{n+1}(z) -
\beta_n A_{n-1}(z).\eqno(S_2)
$$
\end{theorem}

From $(S_1)$ and $(S_2)$, we can derive another identity involving
$\sum_{j=0}^{n-1}A_j$ which is often  helpful:
\begin{equation*}
B_n^2(x) + \textsf{v}\,'(x) B_n(x) + \sum_{j=0}^{n-1}A_j(x) =
\beta_{n} A_n(x) A_{n-1}(x). \eqno(S_2')
\end{equation*}
The conditions $S_1$, $S_2$ and $S_2'$ are usually called the
compatibility conditions for the ladder operators.

\subsection{Analysis of the ladder operators for the semi-classical
Laguerre polynomials} \label{subsec:analysis of lad}

In this section we shall apply the general set-up of the ladder
operators to the polynomials orthogonal with respect to the weight
(\ref{eq:weight}).

For the weight function given in (\ref{eq:weight}), we have
$$\textsf{v}(x)=-\ln w(x)=-\alpha \ln x+x^2-t x ,$$ hence,
$$\frac{\textsf{v}\,'(x)-\textsf{v}\,'(y)}{x-y}=2+\frac{\alpha}{x
y}.$$ It then follows from (\ref{an-def}) and (\ref{bn-def}) that,
if $\alpha>0$,
\begin{equation}\label{eq:an bn}
A_n(x)=2+\frac{R_n}{x}, \quad B_n(x)=\frac{r_n}{x},
\end{equation}
where
\begin{equation}\label{eq:Rn}
R_n=\frac{\alpha}{h_n}\int_0^{\infty}P_n(y)^2
y^{\alpha-1}e^{-y^2+ty}\, dy
\end{equation} and
\begin{equation}
r_n=\frac{\alpha}{h_{n-1}}\int_0^{\infty}P_{n-1}(y)P_n(y)y^{\alpha-1}e^{-y^2+ty}\,
dy.
\end{equation}

Substituting (\ref{eq:an bn}) into $(S_1)$ and comparing the
coefficients of $x^0$ and $x^{-1}$, we have
\begin{align}\label{cond1}
R_n-2\alpha_n+t&=0,
\\
\label{cond2} r_n+r_{n+1}&=\alpha-\alpha_n R_n.
\end{align}
From $(S_2)$ we similarly get two more conditions:
\begin{align}\label{cond3}
1+r_{n+1}-r_n&=2(\beta_{n+1}-\beta_n),
\\
\label{cond4} \alpha_n(r_{n}-r_{n+1})&=\beta_{n+1}R_{n+1}-\beta_n
R_{n-1}.
\end{align}
Finally, relation $(S_2')$ gives us
\begin{align}
\label{cond5} r_n+n &=2\beta_n,
\\
\label{cond6} \sum_{j=0}^{n-1}R_j-tr_n&=2\beta_n(R_{n-1}+R_n),
\\
\label{cond7} r_n^2-\alpha r_n&=\beta_n R_{n-1} R_n .
\end{align}
In particular, it follows from (\ref{cond1}) and (\ref{cond5}) that
\begin{equation}\label{eq:alphan betan}
\alpha_n=\frac{R_n+t}{2},\qquad \beta_n=\frac{r_n+n}{2}.
\end{equation}
It is clear that (\ref{cond3}) is automatically satisfied using
(\ref{eq:alphan betan}).

In Section \ref{sec:Toda}, we showed that
\begin{equation}\label{def:w}
W(t):=q(t/2)=-\frac{\sqrt{2}}{x_n(t)}
\end{equation}
is a solution of the fourth Painlev\'{e} equation with certain
values of the parameters.

Our first objective in this section is to prove
\begin{theorem}If $\alpha>0$, we have
\begin{equation}\label{eq:w=Rn}
W(t)=R_n=\alpha\int_0^{\infty}p_n^2(y)y^{\alpha-1}e^{-y^2+ty} dy,
\end{equation}
where $p_n$ is the orthonormal polynomial associated with
\eqref{eq:weight}.
\end{theorem}

\begin{proof}
With the monic polynomials $P_n$ and orthonormal polynomials $p_n$
defined in \eqref{orthogonality monic} and \eqref{eq:
orthonormality}, respectively, it is easily seen that
$P_n(x)=\sqrt{h_n}p_n(x)$. Hence, by \eqref{recu:orthonormal} and
\eqref{recu:monic}, it follows
\begin{align*}
\alpha_n&=\frac{1}{h_n}\int_0^{\infty}P_n(y)^2y^{\alpha+1}e^{-y^2+ty}\,dy
=\int_0^{\infty}p_n(y)^2y^{\alpha+1}e^{-y^2+ty}\,dy=b_n.
\end{align*}
This, together with the fact that
$$b_n=\frac{1}{2}\left(t-\frac{\sqrt{2}}{x_n}\right)$$ (see Theorem
\ref{thm:recurrence}) and the first equality in (\ref{eq:alphan
betan}), implies
\begin{equation}\label{eq:Rn in xn}
R_n=-\frac{\sqrt{2}}{x_n}.
\end{equation}
The formula (\ref{eq:w=Rn}) is now immediate in view of (\ref{eq:Rn
in xn}), (\ref{def:w}) and (\ref{eq:Rn}).
\end{proof} 

Next we give an alternative proof of Theorem \ref{thm:recurrence}
with the aid of relations established in
\eqref{cond1}--\eqref{cond7}. This, combining with the arguments in
Section \ref{sec:Toda}, will lead to another derivation of the
fourth Painlev\'{e} equation for the recurrence coefficients.

We first show that
\begin{equation}\label{eq:yn in betan}
y_n=2\beta_n-n-\alpha/2.
\end{equation}
To see this, we observe that
\begin{align*}
\beta_n&=\frac{1}{h_{n-1}}\int_0^{\infty}P_{n-1}(y)P_{n}(y)y^{\alpha+1}e^{-y^2+ty}dy\\
&=\sqrt{\frac{h_{n}}{h_{n-1}}}\int_0^{\infty}p_{n-1}(y)p_{n}(y)y^{\alpha+1}e^{-y^2+ty}dy
=\sqrt{\frac{h_{n}}{h_{n-1}}}a_n.
\end{align*}
Therefore, $$\beta_n^2=\frac{h_n}{h_{n-1}}a_n^2.$$ Since it is
easily seen that $\beta_n=h_n/h_{n-1},$ it then follows that
$\beta_n=a_n^2.$ Combing this with the fact that
$y_n=2a_n^2-n-\alpha/2$ gives us (\ref{eq:yn in betan}).

Replacing $x_{n-1}$, $x_{n}$ and $y_n$ in the first equation in
(\ref{eq:recurrence}) by $R_{n-1}$, $R_{n}$ and $\beta_n$ with the
aid of (\ref{eq:Rn in xn}) and (\ref{eq:yn in betan}), it is
equivalent to show that
\begin{equation}\label{eq:first equ}
\beta_nR_{n-1}R_{n}=\left(2\beta_n-n-\frac{\alpha}{2}\right)^2-\frac{\alpha^2}{4}.
\end{equation}
On account of (\ref{cond7}), it is essential to prove
\begin{equation}
\left(2\beta_n-n-\frac{\alpha}{2}\right)^2-\frac{\alpha^2}{4}=r_n^2-\alpha
r_n,
\end{equation}
which is immediate by (\ref{cond5}).

To show the second equation in (\ref{eq:recurrence}), we note that,
again with the aid of (\ref{eq:Rn in xn}) and (\ref{eq:yn in
betan}), it suffices to show that
\begin{equation}\label{eq:sec eq}
2(\beta_n+\beta_{n+1})-2n-1-\alpha=-\frac{1}{2}R_n(t+R_n).
\end{equation}
From (\ref{cond5}) and (\ref{cond2}), it follows that
\begin{equation}
2(\beta_n+\beta_{n+1})-2n-1-\alpha=r_n+r_{n+1}-\alpha=-\alpha_nR_n.
\end{equation}
This, together with (\ref{cond1}), implies (\ref{eq:sec eq}).

\subsection{Isomonodromy deformations approach}

Another easy method to show the connection of the recurrence
coefficients of the semi-classical Laguerre polynomials to the
solutions of the fourth Painlev\'e equation is by using the
isomonodromy deformations approach. The fourth Painlev\'e equation
appears as the result of the compatibility condition $Y_{xt}=Y_{tx}$
of two linear $2\times 2$ systems $Y_x=A(x)Y$ and $Y_t=B(x)Y,$ where
the subscript denotes the partial derivative \cite{JimboMiwa}. Here
\begin{equation}
\begin{aligned}\label{IsoMatix}A(x)= &~\begin{pmatrix} 1 & 0 \\ 0 & -1
\end{pmatrix} x+  \begin{pmatrix}
T & U \\ 2(Z-\theta_0-\theta_{\infty})/U & -T
\end{pmatrix} \\
&+\frac{1}{x} \begin{pmatrix}
 -Z+\theta_0 & -U V/2 \\
2Z(Z-2\theta_0)/(U V) & Z-\theta_0
\end{pmatrix} ,
\end{aligned}\end{equation} where $V=V(T),\;U=U(T),\;Z=Z(T)$ and $V(T)$
satisfies the fourth Painlev\'e equation $(P4)$ with
\begin{equation}\label{IsoPar}
A=2\theta_{\infty}-1,\;\;B=-8\theta_0^2.
\end{equation}
Substituting  (\ref{eq:an bn})  into (\ref{ladder2})  and
(\ref{ladder1}) (in this order) we get the following linear system:
\begin{equation}
\frac{d}{dx}\begin{pmatrix} p_{n-1}(x) \\ p_n(x)
\end{pmatrix} =\left(\begin{pmatrix}
2& 0 \\ 0 & 0
\end{pmatrix} x+
\begin{pmatrix}
-t& -2 \\ 2\beta_n& 0
\end{pmatrix}+\frac{1}{x}
\begin{pmatrix}
r_n-\alpha & -R_{n-1} \\ R_{n}\beta_n & -r_n
\end{pmatrix}
\right)
\begin{pmatrix} p_{n-1}(x) \\ p_n(x)
\end{pmatrix}.
\end{equation}
By replacing the vector $(p_{n-1},p_n)^{tr}$ by
$e^{x(x-t)/2}x^{-\alpha/2}(p_{n-1},p_n)^{tr},$ we get a matrix
similar to (\ref{IsoMatix}). Hence, we can calculate that
$$U=-2,\quad T=-t/2,\quad V=-R_{n-1},\quad Z=\theta_0+\theta_{\infty}-2\beta_n,$$
$$\beta_n=(2\theta_{\infty}-\alpha+2r_n)/4,\quad \theta_{\infty}=(\alpha+2n)/2.$$
Using (\ref{cond1}), (\ref{cond5})  and (\ref{cond7}), we get
$\theta_0^2=(\alpha/2)^2$. Hence, the parameters of the fourth
Painlev\'e equation are $A=\alpha+2n-1$ and $B=-2\alpha^2$. Note
that the fourth Painlev\'e equation is invariant with respect to
scaling: if $q(t)$ is a solution of $\PIV$ with parameters
$\alpha,\;\beta$ and $\lambda^4=1$, then $ \lambda^{-1}q(\lambda t)$
is a solution of $\PIV$ with parameters $\lambda^2 \alpha,\;\beta.$
Hence, the parameters are in agreement with (\ref{parameters}) (with
$n$ replaced by $n-1$).

\subsection{B\"acklund transformations}

In this section we show how to obtain the system
(\ref{eq:recurrence}) from a B{\"a}cklund transformation of the
fourth Painlev\'e equation. First we need a nonlinear relation for
$x_{n-1}$, $x_n$ and $x_{n+1}$. From the second equation of system
(\ref{eq:recurrence})  we get $$y_{n+1}=\frac{\sqrt{2}t x_n-2x_n^2
y_n-2}{2x_n^2}.$$ Using the first equation of this system for index
$n$ and $n+1$, we can eliminate $y_n$ by calculating the resultant
and obtain a nonlinear relation for $x_{n-1},$ $x_n$ and $x_{n+1}$.
Let us denote this cumbersome expression by $E$ for future
reference. Clearly, we can also find an expression between
$q_n,\;q_{n\pm1}$ by using (\ref{change}).

It is known that the fourth Painlev\'e equation $\PIV$ admits a
B{\"a}cklund transformation \cite{GrLSh}. If $q=q(z) $ is a solution
of $\PIV$ with parameters $A$ and $B$, then the function
$$\tilde{q}=T_{\varepsilon,\mu}q=\frac{q'-\mu q^2-2\mu z
q-\varepsilon \sqrt{-2B} }{2\mu q}$$ is a solution of $\PIV$ with
new values of the parameters
\begin{align*}
\tilde{A}=\frac{1}{4}\left(2\mu-2A+3\mu \varepsilon
\sqrt{-2B}\right), \qquad \tilde{B}=-\frac{1}{2}\left(1+A\mu
+\frac{1}{2}\varepsilon\sqrt{-2B}\right)^2,
\end{align*} where
$\varepsilon^2=\mu^2=1.$

\begin{remark}{\rm For our purposes it is sufficient to use the standard
B\"acklund transformations of the Painlev\'e transcendents which are
given in NIST Digital Library of Mathematical Functions  (DLMF
project)\footnote{http://dlmf.nist.gov/32.7}. There are also
algebraic aspects of the Painlev\'e equations. It is known
\cite{Noumi, NoumiYamada} that the B\"acklund transformations of the
fourth Painlev\'e equation form the affine Weyl group of $A_2^{(1)}
$ type. The interested reader can easily  re-formulate our
transformations within the framework of Noumi-Yamada's birational
representation of $\tilde{W}(A_2^{(1)})$.}
\end{remark}

One can verify directly that, for instance, the compositions
$T_{1,1}\circ T_{1,-1}\circ T_{1,1}$ and $T_{1,-1}\circ T_{1,1}\circ
T_{1,-1}$ give rise to the following transformations. Let $q=q_n(z)$
be a solution of $\PIV$ with (\ref{parameters}), then
$$q_{n+1}(z)=\frac{(2\alpha+2zq+q^2-q')(2\alpha-2zq-q^2+q')}{2q(q^2+2zq-q'-4-4n-2\alpha)}$$
is a solution of $\PIV$ with $A=3+2n+\alpha$ and $B=-2\alpha^2.$
Similarly,
$$q_{n-1}(z)=-\frac{(q'+q^2+2zq-2\alpha)(q'+q^2+2zq+2\alpha)}{2q(q'+q^2+2zq-2(2n+\alpha))}$$
is a solution of $\PIV$ with $A=2n+\alpha-1$ and $B=-2\alpha^2.$
After substituting expressions for $q_{n\pm 1} $ into the nonlinear
recurrence relation $E$, we indeed find that this is identically
zero. Hence, we have proved the following statement.
\begin{theorem}
The discrete system (\ref{eq:recurrence}) for the recurrence
coefficients of semi-classical Laguerre polynomials can be obtained
from a B{\"a}cklund transformation of the fourth Painlev\'e equation
$P_{\textup{IV}}$.
\end{theorem}

\subsection{Initial conditions of the recurrence coefficients and classical solutions of the fourth Painlev\'e equation}

It is known \cite{GrLSh} that the fourth Painlev\'e equation admits
classical solutions as follows. Let us take $n=0$, then $\PIV$ with
parameters (\ref{parameters}) for the function $q=q(z)$ has
solutions which satisfy the following Riccati equation
\begin{equation}\label{Riccati}q'+q^2+2zq-2\alpha=0.\end{equation} This equation
can further be reduced to the  Weber-Hermite equation. Using the
change of variables (\ref{change}) we can calculate that
$$x_0(t)=\frac{\sqrt{2}\mu_0}{t\mu_0-2\mu_1}$$ satisfies Eq.~(\ref{Riccati}). Here $\mu_k$ is the $k$th moment $\int
x^kd\mu(x),\;\;k\in\mathbb{N}_0=\mathbb{N}\cup\{0\}$. Moreover,
using the last equality in (\ref{y0}) we get the initial value
$y_0=-\alpha/2$ which coincides with the initial values given in
\cite{BV}. Thus, we have shown that the initial conditions
correspond to classical solutions of the fourth Painlev\'e equation.

The recurrence coefficients $a^2_n$ and $b_n$ can always be written
\cite{Chihara} as ratios of Hankel determinants containing the
moments of the orthogonality measure (see also in \cite{our}).
However, the explicit determinant formulas of classical solutions
for  all the Painlev\'e equations are  known. We refer the reader to
\cite{KajOhta, Murata, NoumiYamada, NoumiOkamoto} and the references
therein for a classification and explicit determinant formulas for
classical solutions (classical transcendental and rational) of the
fourth Painlev\'e equation.

\section{The Freud weight}

In this section, we will study the relation between the recurrence
coefficients of the semi-classical Laguerre weight and the Freud
weight, and show how they are related via the B\"acklund
transformation.

The Freud weight \cite{Freud, Magnus, VanAssche} is given by
\begin{equation*}
w_{\alpha}(x)=|x|^{2\alpha+1}\exp(-x^4+t
x^2),\;\;x\in\mathbb{R},\;\;\alpha>-1.
\end{equation*}
Let us consider the recurrence coefficients of orthonormal
polynomials $\{{\tilde p}_n\}$ with respect to the weight
$w_{\alpha}$.  We have
$$x {\tilde p}_n(x)=A_{n+1}{\tilde p}_{n+1}(x)+A_n {\tilde p}_{n-1}(x).$$ In case of the
monic polynomials $\{{\tilde P}_n\}$ the recurrence relation is
given by
$$x {\tilde P}_n(x)={\tilde P}_{n+1}(x)+A_n^2 {\tilde P}_{n-1}(x).$$ The recurrence
coefficients satisfy
\begin{equation}\label{F1}
4A_n^2(A_{n-1}^2+A_n^2+A_{n+1}^2-t/2)=n+(2\alpha+1)\Delta_n,\;\;
\Delta_n=(1-(-1)^n)/2,
\end{equation}
which is the first discrete Painlev\'e equation d$\PI$.  The
bilinear structure and exact (Casorati determinant) solutions of the
(extended) d$\PI$ equation are studied in \cite{OhtaKaj}.

Let us briefly show how the recurrence coefficients, when viewed as
functions of $t$, are related to the solutions of the fourth
Painlev\'e equation. The differential equation, which can be derived
similarly to the Toda system, is given by
\begin{equation}\label{F2}
\frac{d}{dt}A_n^2=A_n^2(A_{n+1}^2-A_{n-1}^2).
\end{equation}
For simplicity we introduce the notation $f_n(t)=A_n^2(t).$ From
(\ref{F1}), with $n$ and $n-1$, we can find $f_{n-1}$ and $f_{n-2}$.
From (\ref{F2}) we can find $f_{n+1}$. Substituting these
expressions into (\ref{F2}) with $n-1$ we get a second order
differential equation for the function $f_n(t)$. Introducing a new
independent variable $t=2z$ and changing $f(z)=-2f_{n}(t)$ we get
the fourth Painlev\'e equation for the function $f(z)$ with
parameters given by $$ A=-\frac{1}{2}(2+n+4\alpha),\qquad
B=-\frac{n^2}{2},$$ in case $n $ is even and
$$A=\frac{1}{2}-\frac{n}{2}+\alpha,\qquad B=-\frac{1}{2}(1+n+2\alpha)^2,$$
in case $n$ is odd.

For the semi-classical Laguerre weight from Theorem
\ref{thm:recurrence}
$$v_\alpha(x)=x^{\alpha} e^{-x^2+tx},\;\;x>0,\;\;\alpha>-1 ,$$ the
orthonormal polynomials $\{p_n^{\alpha}\}$ satisfy $$x
p_n^{\alpha}(x)=a_{n+1}^{\alpha} p_{n+1}^{\alpha}(x)+b_n^{\alpha }
p_n^{\alpha}(x)+a_n^{\alpha} p_{n-1}^{\alpha}(x). $$ Here we use the
notation in \cite{BV}. It is known \cite{Chihara} that the
polynomials for the Freud weight and the semi-classical Laguerre
weight are related by
$${\tilde p}_{2n}(x)=p_n^{\alpha}(x^2),\;\;{\tilde p}_{2n+1}(x)=x
p_n^{\alpha+1}(x^2)$$ and the following relations holds for the
recurrence coefficients:
\begin{equation}\label{rel1}
a_n^{\alpha}=A_{2n}A_{2n-1}, \qquad
b_n^{\alpha}=A_{2n}^2+A_{2n+1}^2,
\end{equation}
\begin{equation}\label{rel2}
a_n^{\alpha+1}=A_{2n}A_{2n+1}, \qquad
b_n^{\alpha+1}=A_{2n+2}^2+A_{2n+1}^2.
\end{equation}
Since both $A_{2n}^2$ and $b_n^{\alpha}$ in (\ref{rel1}) and
(\ref{rel2}) satisfy the fourth Painlev\'e equation, we are
interested to obtain a relation with the B\"acklund transformation
$T_{\varepsilon,\mu}.$ From (\ref{change}) we get that
\begin{equation}\label{**}q(z)=-2z+2b_n^{\alpha}(t),\qquad t=2z
\end{equation} satisfies the fourth
Painlev\'e equation with (\ref{parameters}). Let us consider case
(\ref{rel1}). We have
\begin{equation}
\label{***}b_n^{\alpha}(t)=A_{2n}^2(t)+A_{2n+1}^2(t)=f_{2n}(t)+f_{2n+1}(t).
\end{equation}
Denoting $$f_1(z)=-2f_{2n}(2z),\qquad f_2(z)=-2f_{2n+1}(2z)$$ we
have that $f_1(z)$ satisfies the fourth Painlev\'e equation with
parameters
$$A=-1-n-2\alpha,\qquad B=-2n^2.$$
The function $f_2(z)$ is also a solution of the fourth Painlev\'e
equation with parameters
$$A=-n+\alpha,\qquad B=-2(1+n+\alpha)^2$$ and
$$f_2(z)=T_{1,-1}f_1(z)=\frac{2n-2z f_1-f_1^2-f_1'}{2f_1}.$$
This, together with (\ref{**}) and (\ref{***}), implies that
\begin{align*}
q(z)&=-2z-f_1(z)-f_2(z)\\
&=\frac{ f_1'-f_1^2-2zf_1-2n}{2f_1}=T_{1,1}f_1(z)
\end{align*} satisfies the fourth Painlev\'e equation with
(\ref{parameters}).

By analogy, in case of (\ref{rel2}) we have that if $f_1(z)$
satisfies the fourth Painlev\'e equation with
$$A=-2-n-2\alpha,\qquad B=-2(n+1)^2,$$ the function
$$f_2(z)=T_{-1,1}f_1(z)=\frac{2+2n-2z f_1-f_1^2+f_1'}{2f_1}$$
satisfies the fourth Painlev\'e equation with
$$A=\alpha-n,\qquad B=-2(n+1+\alpha)^2,$$ then
$$q(z)=-2z-f_1(z)-f_2(z)=T_{-1,-1}f_1(z)$$ is a solution of the fourth Painlev\'e
equation with $$A=2+2n+\alpha,\qquad B=-2(\alpha+1)^2,$$ that is,
with parameters (\ref{parameters}) where $\alpha$ is replaced by
$\alpha+1$.

\section{Discussions}

The recurrence coefficients of semi-classical orthogonal polynomials
are often related to the solutions of Painlev\'e equations $\PII$ --
$\PVI$. Below we include a few examples of such a connection
indicating  the weight, the Painlev\'e equation and the relevant
reference:
\begin{itemize}
\item the weight $e^{x^3/3+tx}$ on $\{x:x^3<0\}$ and $\PII$ \cite{Magnus2};
\item  the weight $x^{\alpha}e^{-x}e^{-s/x}$ and $\PIII$ \cite{ChenIts};
\item   the discrete Charlier weight $w(k)=a^k/((\beta)_k k!),
\;a>0,$  and $\PIII$ (and $\PV$) \cite{GF+Walter2};
\item   $\PIV$ and the weights $|x-t|^{\rho}e^{-x^2}$ in \cite{Chen+Feigin}, $x^{\alpha}e^{-x^2+t
x},\;x>0,$ and  $|x|^{2\alpha+1}e^{-x^4+t x^2}$ (in this paper);
\item    the discrete Meixner weight $(\gamma)_k
c^k/(k!(\beta)_k),\;c,\beta,\gamma>0,$ and $\PV$
\cite{our,GF+Walter1};
\item $\PV$ and the weights $(1-\xi \theta(x-t))|x-t|^{\alpha}x^{\mu}
e^{-x},$ where $\theta$ is the Heaviside function in
\cite{Forrester},
$x^{\alpha}(1-x)^{\beta} e^{-t/x}$ in \cite{chen+dai},
$(1+x)^{\alpha}(1-x)^{\beta}e^{-t x},\;x\in (-1,1),$ in
\cite{BasorChenEhrhardt};
\item $\PVI$ and the weights
$x^{\alpha}(1-x)^{\beta}(A+B\theta(x-t)), x\in [0,1],$ in
\cite{chen+zhang}, $(1-x)^{\alpha}x^\beta (t-x)^{\gamma}, x\in
[-1,1],$ in \cite{Magnus2}; see also \cite{ForreW} for more examples
and applications in random matrix theory.
\end{itemize}

It is an interesting  problem to find new examples of weights
leading to the Painlev\'e equations and their higher order and
multivariable generalizations.

\section*{Acknowledgments}
Part of this work was carried out while GF was visiting K.U.Leuven
for one month. The financial support of K.U.Leuven, MIMUW at the
University of Warsaw and the hospitality of the Analysis section at
K.U.Leuven is gratefully acknowledged. GF is also partially
supported by Polish MNiSzW Grant N N201 397937. WVA is supported by
Belgian Interuniversity Attraction Pole P6/02, FWO grant G.0427.09
and K.U.Leuven Research Grant OT/08/033. LZ is supported by the
Belgian Interuniversity Attraction Pole P06/02.

\end{document}